\documentclass[a4paper,11pt,reqno]{amsart}

\usepackage{latexsym}
\usepackage{url}
\usepackage{amsmath,amsopn,amssymb}
\usepackage{amsthm}
\usepackage{graphics}

\usepackage[margin=0.442in,font=small,labelfont=bf]{caption} 

\newtheorem{theorem}{Theorem}
\newtheorem{lemma}[theorem]{Lemma}

\newtheorem{conjecture}[theorem]{Conjecture}

\newtheorem{proposition}[theorem]{Proposition}

\newtheorem{problem}[theorem]{Problem}

\theoremstyle{remark}

\theoremstyle{definition}

\newcommand{\mfrac}[2]{ { \textstyle{\frac{#1}{#2}} } }

\newcounter{indentedromanrmklistcnt}
\newenvironment{indentedromanrmklist}%
	{\setcounter{indentedromanrmklistcnt}{0}%
	\begin{list}{(\Alph{indentedromanrmklistcnt})}{%
		\usecounter{indentedromanrmklistcnt}%
		\setlength{\topsep}{3pt}%
		\setlength{\itemsep}{3pt}%
		\setlength{\itemindent}{-0.0in}
		\setlength{\leftmargin}{0.5in}} 
	}%
	{\end{list}}%

\newcounter{rmklistcnt}
	{\setcounter{rmklistcnt}{0}%
	\begin{list}{(\arabic{rmklistcnt})}{%
		\usecounter{rmklistcnt}%
		\setlength{\rightmargin}{0.0in}%
		\setlength{\topsep}{3pt}%
		\setlength{\itemsep}{3pt}%
		\setlength{\itemindent}{0.5in}%
		\setlength{\leftmargin}{0.2in}}%
	}%
	{\end{list}}%

\begin{document}


\title{Knights, spies, games and ballot sequences}
\author{Mark Wildon}
\date{\today}
\email{mark.wildon@bristol.ac.uk}
\maketitle

\section{Introduction}

In this paper we solve the \emph{Knights and Spies Problem}:

\begin{quote}
In a room there are $n$ people, each labelled with a unique
number between~$1$ and $n$. 
A person may either be a \emph{knight} or a \emph{spy}. Knights
always tell the truth, while spies may either lie or tell the truth,
as they see fit. Each person in the room knows the identity of everyone
else. Apart from this, all that is known is 
that strictly more knights than spies are present.
Asking only questions of the form:
\begin{center} 
`Person $i$, what is the identity of person $j$?',
\end{center}
what is the least number of questions that will guarantee
to find the true identities of all $n$ people?
\end{quote}

Despite its apparently
recreational character, the Knights and Spies Problem 
is surprisingly deep; it is 
unusual to find such an easily stated
problem that 
can challenge and be enjoyed by professionals and amateurs
alike.\footnote{The author would like to thank Dave Johnson for
telling him about the Knights and Spies Problem in January 2007.\newline
\indent ${}^\star$Part of this work
was 
financially
supported by
the Heilbronn Institute for Mathematical
Research.}
The following remarks introduce some basic ideas
and should clarify its statement.

\subsection{Preliminary remarks}
There is
a simple, if inefficient,
questioning strategy that will
find everyone's identity. 
Assume for the moment 
that $n = 2m$ is even. Given a person $i$, 
if we ask the remaining~$2m-1$ people to state person~$i$'s identity, then the
majority opinion will be correct. For otherwise, the majority consists
of $m$ or more people who have lied, and since only spies can lie, they
must be spies.
With a small extension
to deal with ties in the case when $n$ is odd, this gives
us a  strategy that 
finds everyone's identity in
$n(n-1)$ questions. 

We may refine this strategy by noting that anyone who
ever holds a minority view is a liar. Such people
can be immediately identified as spies, and 
then
ignored as a potential source of information.
Moreover, once we have found
a knight, we may bombard him with questions
to find all the remaining identities. 
However,
even with these improvements, 
the number of questions
required in the worst case is still quadratic in $n$.

When this strategy is followed, the spies are
at their most obstructive when they always
tell the truth. 
This phenomenon will be seen in other contexts below.
We may assume, however, that a spy will lie if
asked about his own identity, and so while it is permitted by 
the rules,
there can be no benefit
in asking a person about themselves. Similarly, there
can be no benefit in asking the same question to any person
more than once.

Before reading any further, the reader is invited to
find a questioning strategy that will use at most
$Cn$ questions for some constant $C$. A hint leading to a
strategy for which $C = 2$ is given in this footnote.\footnote{There
is an inductive strategy that starts by putting people into pairs,
leaving one person out if $n$ is odd, and then asking each
member of each pair about the other}
The optimal~$C$ is revealed in the outline below. 

\subsection{Outline}
We shall solve 
the more general problem, where
it is given that at most $\ell$ spies are present for some $\ell$ 
with
$1 \le \ell < n/2$.
We begin in \S 2 by describing the \emph{Spider Interrogation Strategy}, which
guarantees to 
finds everyone's identity using at most
\[ n + \ell - 1 \]
questions. If, as in the original problem, all we know
is that knights are strictly in the majority, then 
$\ell = \lfloor (n-1)/2 \rfloor$,
%
and so the maximum number of questions asked is 
$f(n)$, where $f$ is defined by
\begin{align*}
f(2m-1) &= 3m-3 \\
f(2m) &= 3m-2. 
\end{align*}
 
No matter which numbers they hold, the spies
can force a questioner following the
Spider Interrogation Strategy
 to ask the full $n + \ell - 1$ questions. If however
the spies are constrained to always lie, or to
always answer `spy', then
usually fewer questions are required. We
determine the probability distribution of the
number of questions asked; 
remarkably it
is the same in either case. The proof is bijective,
using two lemmas related to the well-known  ballot counting problem
(see \cite[III.1]{Feller}).

In  \S 3 we prove that 
any questioning strategy will, 
in the worst case, require at least $n + \ell - 1$
questions. 
Hence the answer to our original problem is
that, provided $n \ge 3$, the smallest number of questions
that can guarantee success is $f(n)$. Since
\[ 
0
 \le 3n/2 - f(n) \le 2 \]
for all natural numbers $n$, 
it follows that the optimal constant $C$ is $3/2$.
The proof in \S 3 is presented in terms of an optimal
strategy for the second player in the
two-player game in which the first player poses
questions (in the standard form), and the second supplies
the answers (`knight' or `spy'), with the aim of forcing
her opponent to ask at least $n + \ell - 1$ questions before
she can be sure of everyone's identity. This game
provides a setting for all the problems considered in this paper.

It is natural to ask whether there is
a questioning strategy which never uses more than
$n + \ell - 1$ questions, and will with reasonable probability use
fewer, no matter how cleverly the spies
answer. 
In \S 4 we modify the Spider Interrogation
Strategy to show that such a strategy exists
in the case when~$\ell$ is at most~$\sqrt{n}$.
(Of course, given the result of \S 3, there is always be a
non-zero probability that 
the full number of questions will be required.)
 We then present some evidence for the conjecture 
that such a strategy exists
for all admissible values of $\ell$. 
We end in \S 5
by briefly discussing two further open problems.

\section{The Spider Interrogation Strategy}

\subsection{Description}
The \emph{Spider Interrogation Strategy} has four steps: 
the first step,
in which we hunt for someone who we can guarantee
is a knight, is the key to its workings.
We suppose that at most 
$1 \le \ell < n/2$ of
the $n$ people in the room are spies.

\subsubsection*{Step~1}
Choose any person as a \emph{candidate}. 
Repeatedly ask new people about the candidate until \emph{either}

\begin{indentedromanrmklist}
\setlength{\parindent}{0pt}
\item[(a)]
strictly more people have said that the candidate  
is a spy than have said that he is a knight, \emph{or}

\item[(b)]
$\ell$ people have said that the candidate is a knight.
\end{indentedromanrmklist}

If we end in case (a), with the candidate accused by $a$
different people, then he must have been supported 
by $a-1$
different people. 
Whatever his true identity, it is easily checked that
at least $a$ of the~$2a$ people involved 
are spies. Hence 
if we reject the candidate,
ignore all $2a$ of the people involved so far, 
and replace $\ell$ with~$\ell-a$, we may repeat Step 1
with a smaller problem. Eventually, since spies are in
a strict minority, we must finish in case (b). The
successful candidate is supported by~$\ell$ people,
so must be a knight.

\subsubsection*{Step 2}
Let person $k$ be the 
knight found at the end of Step~1.
All future questions
will be addressed to him. In this step, use him to
identify each person who has not yet been involved
in proceedings, and also each of the rejected candidates from Step~$1$.

\subsubsection*{Step 3}
Let persons $m_1$, \ldots, $m_t$ be the rejected
candidates whose identities were determined in Step~2. 
Suppose that person $m_i$
was accused by $a_i$ people. 


\begin{indentedromanrmklist}
\setlength{\parindent}{0pt}
\item[(a)] If person $m_i$ is a knight, then the $a_i$
people who accused him are spies. 
Identify the $a_i-1$ people who supported 
him.

\item[(b)] If person $m_i$ is a spy, then the $a_i-1$ people
who  supported him are spies. 
Identify the $a_i$ people who accused him. 
\end{indentedromanrmklist}

\subsubsection*{Step 4} 
Finally, identify each person who supported 
person~$k$'s candidacy. Since the people who accused person $k$ must be spies,
everyone's identity is now known.

\medskip

It will  be useful to represent the progress of the
Spider Interrogation Strategy
by a labelled digraph on the set $\{1,2,\ldots,n\}$ in which
we draw an edge from vertex~$i$ to vertex~$j$ if person~$i$ 
has been asked about person~$j$, and label it with
person~$i$'s answer.
We shall refer to
such a graph as a \emph{question graph}. 
Figure~1 overleaf 
shows a typical
question graph after Step~1 of the Spider
Interrogation Strategy. 
Its characteristic structure gives the Spider Interrogation
Strategy its name.

\begin{figure}[h!]\label{fig:spiderex}
\begin{center}
\includegraphics{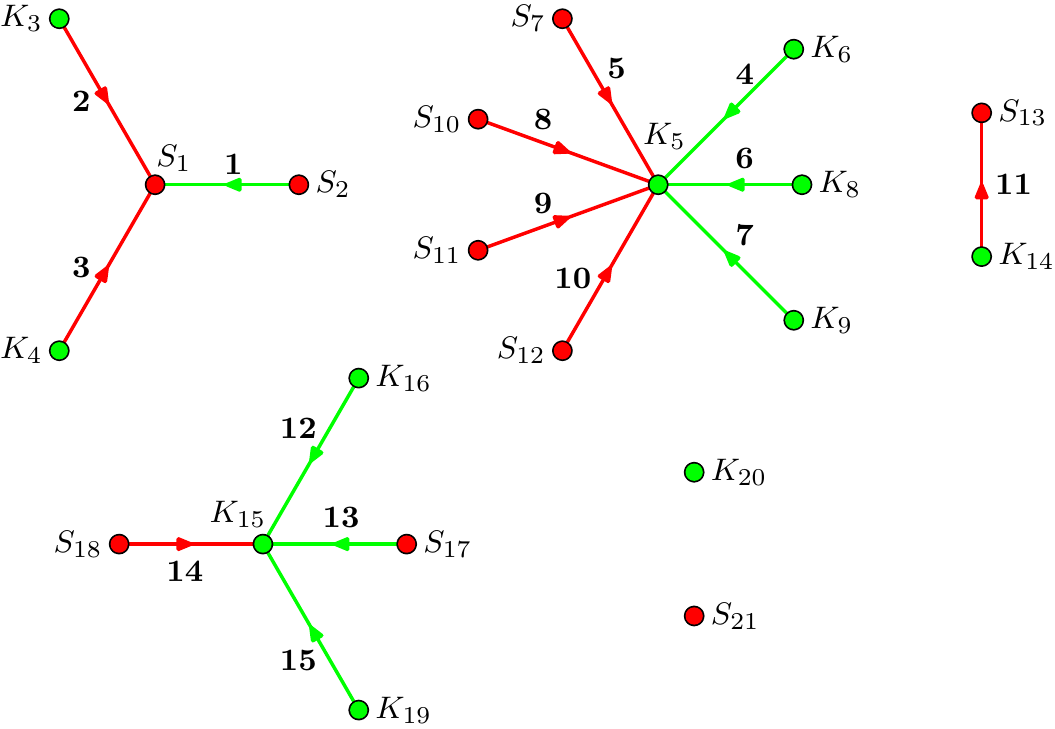}
\end{center}
\caption{The question graph at the end of Step 1
of the Spider Interrogation Strategy in a $21$-person room
with $\ell = 10$. Green arrows show supportive statements and red
arrows show accusations. Questions are numbered in bold.
The  candidates are $S_1$ (rejected), $K_5$ (rejected), $S_{13}$ (rejected) 
and $K_{15}$
(successful).
Spies are assumed to lie in all their answers, except for 
$S_{17}$, who we suppose answers truthfully when asked
about $K_{15}$.  All future questions will
be addressed to the knight~$K_{15}$. For instance, in Step~2 
he will be 
asked about $S_1$, $K_{5}$, $S_{13}$, $K_{20}$ and~$S_{21}$.
The total number of questions asked is~$29$.}
\end{figure}

An
 interesting feature of the Spider Interrogation Strategy, 
already visible in Figure~1, 
is that 
 it guarantees that 
 each spy in the room will be asked
at most one question. 

\medskip 

\subsection{On the number of questions asked}
It is not hard to show that the Spider Interrogation Strategy uses
at most $n + \ell - 1$ questions. 
In fact we can easily prove something
more precise.

\begin{proposition}\label{prop:sis}
The total number of questions asked by a questioner following
the Spider Interrogation Strategy is
\[ n + \ell - 1 - r \]
where $r$ is the number of knights rejected as candidates
in its first step.
\end{proposition}

\begin{proof}
After Step 2 is complete, the underlying graph of the question
graph is a tree. Therefore $n-1$ questions have been asked by this point.
The number of questions asked in Step~3 is
$a_1 + \cdots + a_t - r$.
The knight $k$ was accepted after 
$\ell - (a_1 + \cdots + a_t)$ people supported 
him, hence the total number
of questions asked in Steps~3 and~4 is $\ell - r$. The result follows.
\end{proof}

Thus a questioner following the Spider Interrogation
Strategy 
saves one question from the maximum of $n + \ell - 1$
every time a knight is rejected as a candidate. The spies
can easily make sure this never happens, most simply by always
answering
truthfully. If however the spies always lie, or always 
answer `spy',
then is is probable that fewer questions will be required.
We shall refer to these behaviours as \emph{knavish}
and \emph{spyish}, respectively. 

\begin{theorem}\label{thm:strat}
Suppose that there are $k$ knights and $s$ spies randomly
arranged in the room, and that the spies are constrained to act
either knavishly or spyishly.
The probability that a questioner following the Spider Interrogation Strategy 
asks exactly $q$~questions is independent
of the constraint on the spies.
In either case, the expected number of questions
saved is
\[ \frac{1}{\binom{k+s}{s}} \sum_{r=0}^{s-1} \binom{k+s}{r}.\]
\end{theorem}

In particular, if $k = s+1$ then the sum of binomial coefficients 
is $2^{2s} - \binom{2s+1}{s}$, and it follows from Stirling's formula
that the number of questions saved is $\mfrac{1}{2}\sqrt{\pi s} - 1 + o(s)$.
Hence, 
in a large room in which knights are only just in the majority,
a questioner following the Spider Interrogation Strategy can
expect to ask about
\[ \frac{3n}{2} - \sqrt{\frac{\pi}{8}} \sqrt{n}  \]
questions. Another asymptotic result worth noting is that
when $k = 2s$, the sum of binomial coefficients agrees
in the limit  with $\binom{3s}{s}$, and so the expected
number of questions saved tends to~$1$ as $s$ tends to infinity.

Our proof of Theorem~\ref{thm:strat} is bijective, and does not
give an explicit formula for the probabilities involved. (Indeed,
it seems unlikely that any simple such formula exists.)
Some idea of how these  probabilities vary is given by
Figure~10 at the end of \S 4, which shows the results from a computer
simulation of rooms with $51$~knights and $49$~spies. 

\subsection{Paths}
We shall represent 
the sequence of questions asked 
in Step~1 
of the Spider Interrogation Strategy 
by a path
in which we step up every time
a knight is supported or a spy is accused,
and down every time a knight is accused or a spy is
supported. An initial step, which could be thought of
as the candidate implicitly voting for himself, is taken whenever
a new candidate is chosen. 

\begin{figure}[b]\label{fig:spiderpath}
\begin{center}
\includegraphics{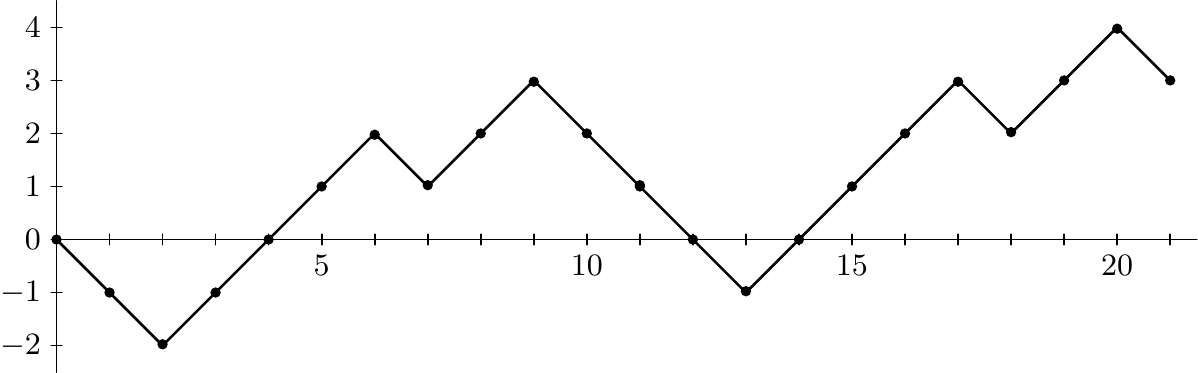}
\caption{The path 
corresponding to Step~1 of the Spider Interrogation Strategy in the
$21$-person room shown in Figure~1. The final two steps 
correspond to extra questions asked to the knight $K_{20}$
and the spy $S_{21}$ about the successful candidate $K_{15}$.
(We have supposed that $S_{21}$ lies.)} 
\end{center}
\end{figure}

Rather than end the path
when a candidate is accepted, we instead imagine that we 
continue to question people about
our accepted candidate until everyone in the room
has either been a candidate, or has been asked a question.
Thus our paths will always have exactly $n$ steps.
  We give each path with a given number of upsteps
and downsteps
the same probability; our extension of paths therefore 
mimics Fermat's
solution of the famous 
\emph{Probl{\`e}me des Points} (see \cite[page~300]{RouseBall}
for an accessible account).
Figure~2 shows the path corresponding to the $21$-person
room in 
Figure~1.

We say that a path \emph{visits $m$ from above at time $r$} if
its height after~$r$ steps is~$m$, and 
its $r$-th step
is downwards. Thus the path shown in 
Figure~2 
visits~$1$
from above exactly twice, at times~$7$ and~$11$. 

\begin{lemma}\label{lemma:annoy}
Let $P$ be a path representing the questions asked in
Step~1 of the Spider Interrogation Strategy. There is a bijective
correspondence between visits of $P$ to $0$ from
above and 
 rejected knights in this step.
\end{lemma}

\begin{proof}
It suffices to prove that,
once a candidate
has been accepted, the path never returns to $0$.
This is left to the reader as a straightforward exercise.
\end{proof}

We need two further probabilistic lemmas on paths, each
of some independent interest.

\begin{lemma}\label{lemma:ballot}
Let $k \ge s$ and let $p \ge 0$.
The probability that a path with $k$ upsteps and~$s$ downsteps
visits $m$ from above exactly $p$ times
is constant for \hbox{$-1 \le m \le k-s$}.
\end{lemma}

\begin{proof}
Let $0 \le m \le k-s$.
We shall show
that the probabilities agree for~$m-1$ and $m$.
Let $P$ be a path with $k$ upsteps and $s$ downsteps. 
Suppose that the first time $P$
visits~$m$ is after step $b$, and that the
last time $P$ visits~$m$ is after step $c$.
(Since $m \ge k-s$, $b$ and $c$ are well-defined.)
Reflecting
the part of~$P$ between $b$ and~$c$
in the line~$y = m$ gives
a new path, $P'$. Figure~3 
shows this reflection when $m = 1$ for the path in Figure~2.

\begin{figure}[h!]\label{fig:ballot}
\begin{center}
\includegraphics{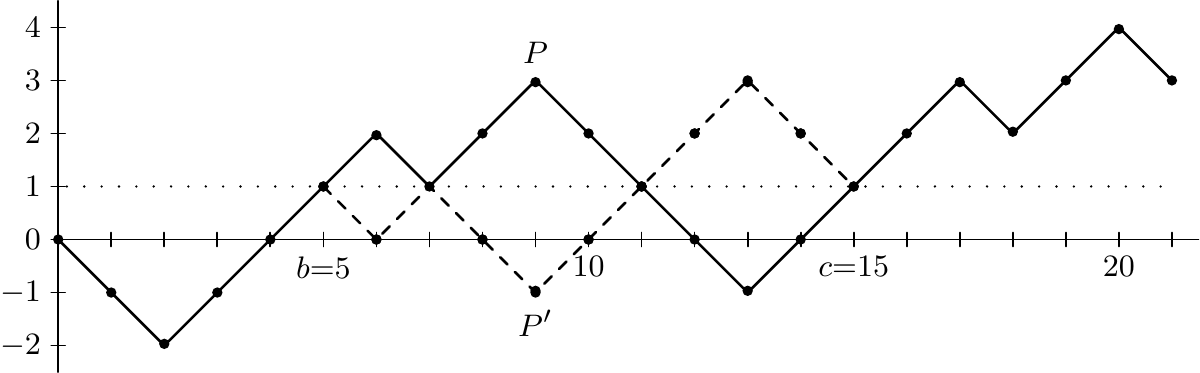}
\end{center}
\caption{The path $P'$ is the reflection of $P$ in the line $y=1$ between
$b=5$ and $c=15$.}
\end{figure}

One easily sees that 
$P$ visits $m$ from above exactly as many
times as $P'$ visits~$m-1$ from above. Similarly
$P$ visits $m-1$ from above exactly as many
times as $P'$ visits $m$ from above.
The result follows.
\end{proof}

\begin{lemma}\label{lemma:refl}
Let $k \ge s$.
The expected number of visits to~$-1$ from above
for a path with~$k$ upsteps and~$s$ downsteps is
\[ \frac{1}{\binom{k+s}{s}} \sum_{r=0}^{s-1} \binom{k+s}{r}.\]
\end{lemma}

\begin{proof}
Let $c(k,s)$ be the total number of times all paths
with $k$ upsteps and~$s$ downsteps visit $-1$
from above. We must prove that
\[ c(k,s) = \sum_{r=0}^{s-1} \binom{k+s}{r} .\]
We work by induction on $s$. If $s=0$ then it is impossible
for any path to visit $-1$, so the result obviously holds in this case.

For the inductive step we 
use reflection in a slightly different way, which is, in fact,
the standard way it is used.\footnote{Feller \cite[Chapter~3]{Feller}, 
gives a good introduction
to this reflection argument and its possible applications.}
Let $P$ be a path with~$k$ upsteps and~$s$ downsteps which visits $-1$
from above at least once. If~$P$ visits $-1$ for the first time
after step~$d$, then reflect the part of $P$ between $0$ and $d$ in
the line~$y=-1$. 
This gives a new path $P^\star$ from 
$(0,-2)$ to $(k+s,k-s)$, as shown in Figure~4 
below. 

\begin{figure}[h!]\label{fig:refl}
\begin{center}
\includegraphics{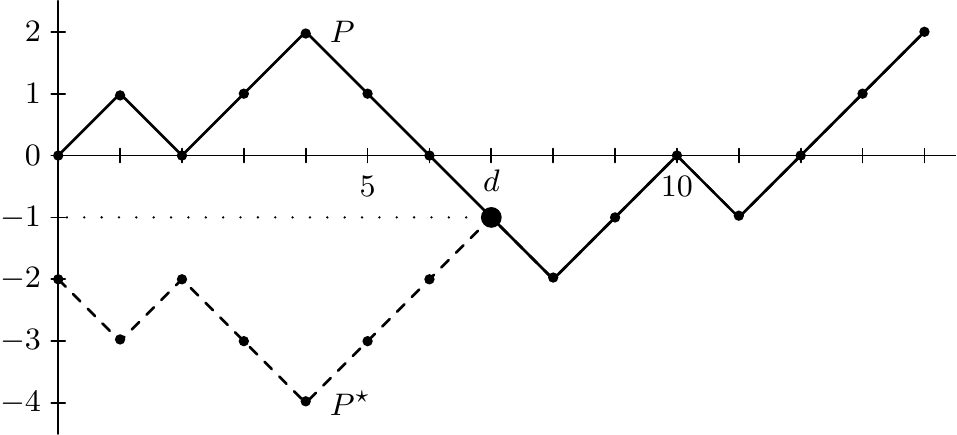}
\end{center}
\caption{The path $P$, which visits $-1$ for the first
time after step $d=7$, is reflected 
to the path $P^\star$
starting at~$(0,-2)$.}
\end{figure}


If $P$  visits~$-1$ from
above exactly $m$ times then~$P^\star$ visits $-1$
from above exactly $m-1$ times. 
Since there are $\binom{k+s}{s-1}$
possible paths $P^\star$ from $(0,-2)$ to $(k+s,k-s)$,
each with $k+1$ upsteps and $s-1$ downsteps, we have
\[ c(k,s) = \binom{k+s}{s-1} + t \]
where $t$ is the total number of times all paths from $(0,-2)$
to \hbox{$(k+s,k-s)$} visit~$-1$ from above. Each such
path has $k+1$ upsteps and $s-1$ downsteps. Shifting to $(0,0)$
and applying Lemma~\ref{lemma:ballot} we see that $t = c(k+1,s-1)$.
The lemma now follows by induction.
\end{proof}

\subsection{Proof of Theorem~\ref{thm:strat}}

The first part of this theorem asserts
that if there are $k$~knights and 
$s$ spies in the room, then the probability
that exactly~$q$ questions are saved is 
independent of whether spies act knavishly 
or spyishly. As in Lemma~\ref{lemma:refl},
we shall work by induction on~$s$.

The two behaviours for the spies differ
only when a spy is asked about another spy, so when
$s = 0$ or $s=1$, the probabilities agree.
When $s \ge 2$ we may use induction to
reduce to the case where the first candidate
is a spy, and the first question is asked to another spy.

Suppose first of all that spies behave knavishly.
Then, in a path corresponding to Step~1 of the Spider Interrogation Strategy,
questions asked to knights correspond to upsteps,
and questions asked to spies correspond to downsteps,
and the first two steps are downwards.
By Lemma~\ref{lemma:annoy}, the probability
that exactly $p$ knights are rejected is equal to the probability
that a path  with $k$ upsteps and $s-2$ downsteps
visits~$2$ from above exactly $p$ times.

Now suppose that spies behave spyishly.
In this case our initial candidate is rejected at question $2$, and
we choose a fresh candidate.
By induction,
we may assume that all the remaining spies in the room behave knavishly.
The remaining questions in Step~1 are 
represented by a path with $k$ upsteps and~\hbox{$s-2$} downsteps.
Hence, the probability that exactly $p$ knights are rejected
is the probability that 
 a path
with $k$ upsteps and~\hbox{$s-2$} downsteps visits~$0$ 
from above exactly $q$ times.

By Lemma~\ref{lemma:ballot}
these two probabilities 
are equal.
Moreover, by Lemma~\ref{lemma:refl},
the expected number of visits
to~$0$ from above of a path with~$k$ upsteps
and~$s$ downsteps is
\[ \frac{1}{\binom{k+s}{s}} \sum_{r=0}^{s-1} \binom{k+s}{r}.\]
When spies act knavishly, this is the expected number of 
questions saved. 
We have just seen that the behaviour of the spies
does not affect the distribution of this quantity, so this is also 
its expected value when spies act spyishly. 
This completes
the proof of Theorem~\ref{thm:strat}.

\section{A lower bound}
In this section we shall prove that
any questioning strategy, will, in the worst case, require
at least
$n + \ell - 1$ questions to find everyone's true identity.

The difficult we face in proving this result is that we must
somehow take into account \emph{every} possible
questioning strategy that may be employed, irrespective
of how bizarre it might seem. This is much the same problem
that confronts a player of a game such as chess or \emph{go},
and so it is perhaps not surprising that it is very helpful
to think of our problem in this context.

\newcommand{\negvthinspace}{\hspace{-1pt}}

\subsection{A mathematical game} 
The game of \emph{`Knights and Spies\negvthinspace'} is played between two players:
an \emph{Interrogator} and a \emph{Secret-Keeper}.
At the start
of the game the players agree on values for the usual parameters $n$
and~$\ell$, with as usual $1 \le \ell < n/2$. 

In a typical turn the Interrogator poses a question (in the
standard form)
to the Secret-Keeper. The Secret-Keeper considers the various ways
in which knights and spies can be arranged in the room 
and then supplies the
answer: `knight' or `spy'. 
The Interrogator's aim 
is, of course, to determine everyone's identity. The Secret-Keeper 
acts as the agent of malign fate and aims to answer
in a way that will inconvenience
the Interrogator as much as possible.

If, 
at the beginning of a turn, the Interrogator believes that she
is certain of everyone's identity, she may \emph{claim} by giving
the full set of people who she believes are spies. 
The Secret-Keeper must then either
\emph{refute} her claim, by exhibiting a different set that is also consistent
with her answers so far, or agree that the secret is out. 
The Interrogator wins if she makes a successful
claim before turn~\hbox{$n + \ell$}, and draws if she makes a successful claim
at the start of turn~\hbox{$n + \ell$} (after asking $n+\ell-1$ questions). In any other event victory goes to
the Secret-Keeper.\footnote{Practical experience 
suggests that it is all too easy for
 the Secret-Keeper
to inadvertently answer in such a way that  
all consistent interpretations of her answers
require strictly more than $\ell$ spies to be present.
Such errors may be
avoided by using the author's program \emph{Gamechecker}, which
makes an exhaustive search for an assignment of identities consistent with 
the Secret-Keeper's responses.
It reports if there is a unique such assignment, so it can also
be used to adjudicate claims by the Interrogator. 
The Haskell source code for \emph{Gamechecker} is available
from the author's website: \url{http://www.maths.bris.ac.uk/~mazmjw}.
It
would be interesting to know if there is a polynomial time
algorithm for deciding whether an incomplete game is in a
consistent state; the back-tracking algorithm used by \emph{Gamechecker}
works well in practice, but in the worst case requires exponential time and space.}

Note that the Secret-Keeper is not committed, even privately, to any 
particular arrangement of knights and spies. 
All that matters is that, 
at every point in the game, there is a way to assign identities to 
the people in the room that is consistent with her answers so far,
and with 
the requirement that at most $\ell$ spies are present.
The small-scale game shown in Figure~5 
should clarify this point.

\vspace{-6pt}
\begin{figure}[h!]\label{fig:exgame}
\begin{center}
\includegraphics{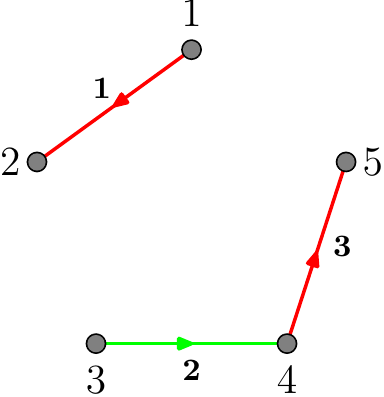}
\end{center}
\caption{The question graph part way through a
novice game in a $5$ person room with $\ell = 2$. Green arrows show supportive statements,
red arrows show accusations.
For instance, in the first turn the
Interrogator asked person 1 about person 2, and the Secret-Keeper accused by replying `spy'.}
\end{figure}

The Secret-Keeper's third reply in this game was a 
blunder, for after it, the Interrogator, reasoning that at
most two spies are present, can be sure that person~$5$ is a spy, 
and also that either person~$1$ or person~$2$ is  a spy. She will therefore
be able to claim after just one more question.
If the Secret-Keeper had instead supported by replying `knight' on
her third turn, then the Interrogator can be held to the target of
six questions; the reader may check that it
is the Secret-Keeper's choice whether
one or two spies appear in
the Interrogator's eventual claim. 

Our required result, that any questioning strategy will,
in the worse case, require at least $n + \ell - 1$ questions,
is equivalent to the following theorem.

\begin{theorem}\label{thm:game} The Secret-Keeper 
has a strategy that ensures 
the Interrogator cannot claim before
she has asked $n + \ell - 1$ questions.
\end{theorem}

We refer the reader to \cite[\S 10.1]{vonNeumann} for
a formal axiomatisation of two-player games which is 
more than capable of expressing Theorem~\ref{thm:game}.

\subsection{The Mole Hiding Strategy}
We prove Theorem~\ref{thm:game} by showing that
the following two-phase strategy for the Secret-Keeper
(referred to as the \emph{Mole Hiding Strategy}) will hold the Interrogator
to $n + \ell - 1$ questions. For simplicity, we shall assume that
the Interrogator never 
repeats a question verbatim or asks someone to state his own identity;
 the discussion
in \S 1.1 tells the Secret-Keeper how to reply to such questions,
and shows that this
is not a significant restriction.

\subsubsection*{Phase~1.}
Answer the first $\ell-1$ questions posed by the Interrogator
with blanket accusations. Let $G$ be the subgraph of
the question graph whose vertices correspond to people who
have already been involved in one of the first~$\ell-1$~questions.
Suppose that the underlying
graph of $G$ is the union of the connected components
$G_1, \ldots, G_c$. Let $G'$ be the set of people who have
not yet been involved in proceedings.

\subsubsection*{Phase~2.} Now answer the Interrogator's
questions according to the following rule.
Suppose that
the Interrogator's question asks for 
the identity of person~$j$.
If~$j$ belongs to~$G'$ then support, and
if~$j$ belongs to the component~$G_i$ 
then accuse, \emph{unless} in Phase~2 of the game
the Interrogator has already
asked about everyone else in~$G_i$; in this case, support.

\medskip
An example game in which $n = 12$ and $\ell = 5$ is shown in 
Figure~6 overleaf.  
The subgraph $G = \left\{1,2,3\right\} \cup \left\{4,5\right\}$ 
has two connected components. 
The Interrogator can be sure after
$16$ questions that the only spies present are persons~$2$,~$3$ and~$5$,
but is unable to claim any earlier; the game therefore
ends in a draw.

\begin{figure}[h!]\label{fig:bicyclegame}
\medskip
\includegraphics{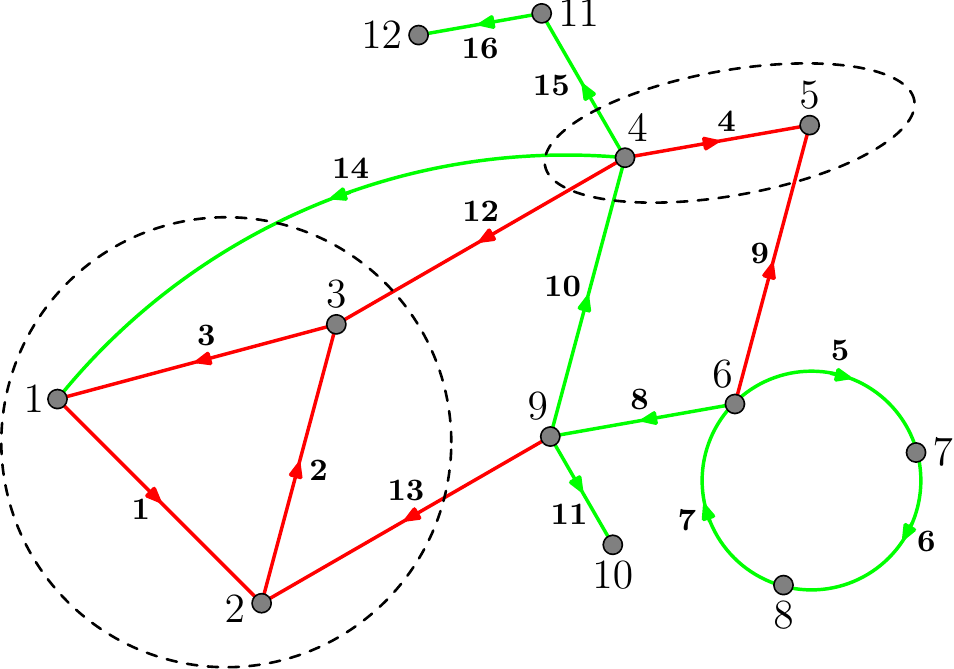}
\caption{A game in a $12$ person room with $\ell = 5$. The
Secret-Keeper adopts the Mole Hiding Strategy, and holds
the Interrogator to a draw. Questions are numbered in bold. 
The connected
component of the subgraph~$G$ 
are marked.} 
\end{figure}


The name of this strategy comes from the 
Interrogator's time-consuming search
through $G'$ for  hidden spies, and through~$G$ 
for hidden knights.
The proof of the following proposition shows
that this search
is unavoidable.

\begin{proposition}\label{prop:game}
If the Secret-Keeper follows the Mole Hiding 
Strategy then, at every point in the game, there is a subset of
people that
can consistently be the set of spies in the room. Moreover,
at the 
beginning of each turn $t$ with $t \le n + \ell - 1$, there
are two different such subsets.
\end{proposition}

\begin{proof}
Suppose we are at the start of turn~$t$.
Since extra questions can only increase the requirements
a consistent assignment of identities to the people in
the room must satisfy, we 
may assume without loss of generality that~$t \ge \ell - 1$.
Hence the subgraph $G$ is defined.

For each component $G_i$ of $G$, if
the Secret-Keeper
has already asked about everyone in $G_i$, then
let person~$k_i$ be the unique person who has
been supported in Phase~2 of the game. Otherwise,
choose for $k_i$ 
any person in~$G_i$ who has not
yet been asked about.
Let
\[ S = G \setminus \left\{ k_1, \ldots, k_c \right\},\]
and let $K$ be the complement,
\[ K = G' \cup \left\{ k_1, \ldots, k_c \right\}. \]

Let $k \in K$ and let $y \in \{1,2,\ldots,n\}$. If the Secret-Keeper has told the Interrogator
that person~$k$ supports person~$y$, then this question
must have occurred in Phase~2 of the game, and either
$y \in \{k_1, \ldots, k_c\}$ or~$y \in G'$. Hence $y \in K$.
Similarly, if the Secret-Keeper
has told the Interrogator that person~$k$ accuses person~$y$, then
$y \in S$. Hence, provided that $S$ is not too large,
the Secret-Keeper's answers
are consistent with~$S$ being the full set
of spies.

Suppose that the connected component $G_i$ contains $v_i$~people
and has~$e_i$ edges. 
The number of questions
asked in Phase~1 of the game is $e_1 + \cdots + e_c = \ell -1$. 
By a standard
result, $e_i \ge v_i - 1$, and hence
\begin{align*} 
|S| &= (v_1-1) + (v_2 -1) + \cdots + (v_c-1)  \\
    &\le e_1 + \cdots + e_c \\
    &= \ell -1.
\end{align*}
Therefore we even have one spy left to play with.

Now suppose that $t \le n + \ell - 1$. At most $n-1$ questions have been
asked in Phase~2 of the game, so there is some person, say person $x$,
who has not been asked about in this phase. We may assume that
if $x$ belongs to $G$, say with $x \in G_i$, then we chose $k_i = x$.
We shall use person $x$ to construct a set~$S^\star$, different from~$S$,
that can also be taken as the set of spies. There
are two cases to consider.

If $x \not\in S$ then let $S^\star = S \cup \{x\}$. By our choice
of $x$, person $x$ has never been supported by anyone in the room,
so it is consistent that he is a spy.
Since $|S^\star| = |S| + 1 \le \ell$,
it is consistent that $S^\star$ is the set of spies.

If $x \in S$ then let $S^\star = S \setminus \{x \}$.
Person $x$ has only been accused by people in $S^\star$. Moreover,
one easily checks that person $x$ has accused only people in~$S^\star$,
and supported only people not in $S^\star$. Hence it is consistent that 
person~$x$
is a knight, and that~$S^\star$ is the set of spies.
\end{proof}

\vbox{
It follows from the first part of Proposition~\ref{prop:game}
that the Secret-Keeper can adopt the Mole Hiding Strategy without
breaking the rules of the game. The second part shows that the 
Interrogator will be unable to claim before she has asked
$n + \ell - 1$ questions.
Theorem~2 is an immediate corollary.}

\subsection{Final remarks on the game}
We end this section with two remarks on the game
we have introduced, each with a hint of the paradoxical.

Firstly, the author's experience is
that 
most players expect to find it easier to play as the 
Secret-Keeper than the Interrogator, but, to their surprise, find
that after the first few games, the reverse is true. Since it is far from
obvious that $n + \ell - 1$ questions suffice, this
seems somewhat remarkable. 

Secondly we note that
the Mole Hiding Strategy is optimal
(in the game-theoretic sense) since it guarantees
to hold the Interrogator to \hbox{$n + \ell - 1$} questions, which,
given the existence of the Spider Interrogation Strategy, 
is the best the Secret-Keeper can hope for. 
This is not to say however, that the Mole Hiding Strategy
cannot be improved. Its defect is that it does not
punish bad play on the part of the Interrogator as harshly
as is possible.

For example, in the game shown in Figure~6, 
the Interrogator's third question was in fact a blunder, after
which the Secret-Keeper can, by extending Phase~1 of the
game for an extra question, force the Interrogator
to ask~$17$ questions. 
This changes the outcome of the game
from a draw into victory for the Secret-Keeper.
More generally, if the 
Secret-Keeper is willing to depart
from the strict letter of the Mole Hiding
Strategy, she
can win any game in which the Interrogator's
questions during Phase~1 form an undirected cycle.
It would be interesting
to know what other early plays by the Interrogator can be punished.

\section{Cycles and chains}

In \S 2 we noted that, no matter how the spies are arranged in the
room, they can ensure that a questioner following the Spider Interrogation
Strategy asks~$n + \ell - 1$ questions.
It is natural to ask
whether there is a questioning strategy which never uses more
than $n + \ell - 1$ questions, and also has 
a reasonable probability
of using fewer, no matter how cleverly the spies answer.

\subsection{A partial result}
When~$\ell$ is small compared to~$n$ this question---in
one interpretation at least---has an affirmative
answer. This can be shown by modifying the Spider Interrogation
Strategy; we give the required changes in outline.

\subsubsection*{Step~1}
Ask
person~$1$ about person~$2$, then person~$2$
about person~$3$, and continue in this manner, until either we meet
an accusation, or we have asked~$\ell$ questions. In the latter case,
person~$\ell + 1$ must be a knight. 
If we simply ask him
about everyone else in the room, then we find everyone's identity
in $n + \ell - 1$ questions. Moreover, if 
we begin by asking about person~$1$
then, in the event that he transpires to be a knight, the resulting
cycle in the question graph tells us that the first $\ell+1$ people
are all knights. A further $n-(\ell+1)$ questions
find all the remaining identities, giving a total of just~$n$ questions.

In the former case, suppose that
person~$t$ accused person $t+1$.
If \hbox{$t=1$}, then we have not yet departed
from the normal Spider Interrogation Strategy.
If~\hbox{$t > 1$} then
treat person~$t$ as a candidate who
has been supported by~$t-2$ people, and continue to question
new people about him. If eventually he is rejected,
after having been accused by~$a$ different people, then the resulting 
spider contains $2(a+1)$ people,
of whom at least~$a+1$ are spies. 
The threshold for acceptance
of the next candidate is therefore $\ell - (a+1)$. Now follow
Step~1 of the unmodified strategy.

\subsubsection*{Steps 2,3 and 4} These are analogous to the unmodified
strategy. The reader may check that,
 once
the identity of person~$t$ has been determined,
$a+1$ questions 
suffice to find all
the identities of the people in the first spider.
It therefore follows, along similar lines to Proposition~\ref{prop:sis},
that it is possible to determine 
everyone's identities in~$n + \ell - 1$ questions.
Figure~7 below shows an illustrative example.

\begin{figure}[h!]
\begin{center}
\includegraphics{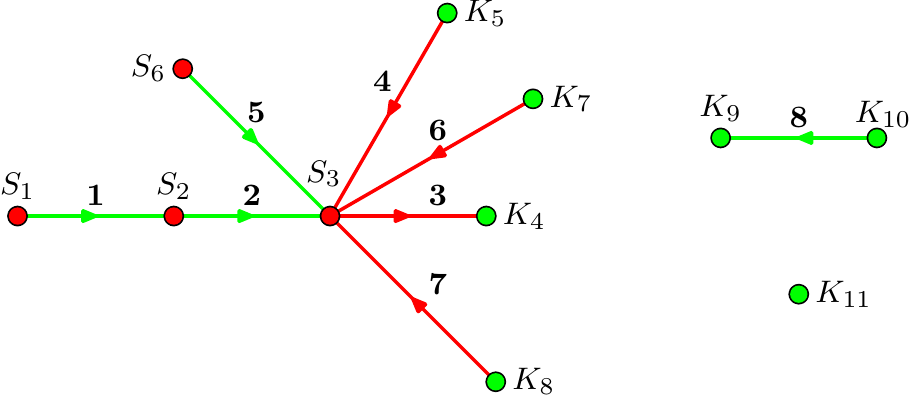}
\end{center}
\caption{The end of Step~1 in the
 modified Spider Interrogation Strategy in an $11$ person
room with $\ell = 5$, in which spies act knavishly. 
The first candidate $S_3$ is rejected, and
the second $K_{9}$ is accepted. In Step~2, the knight $K_9$ will be asked
about $S_3$ and $K_{11}$, and
in the modified version of Step~3, he will be asked about 
his fellow knights, $K_4$, $K_5$, $K_7$, $K_8$ and $K_{10}$.
The full~$15$ questions are required.}
\end{figure}
 
The event that none of the first $\ell+1$ people in the room
is a spy has probability at least
\[ g_\ell(n) = \Bigl( 1 - \frac{\ell}{n-\ell}\Bigr)^{\ell + 1}. \]
For fixed $\ell$, the lower bound $g_\ell(n)$ is an increasing function of $n$.
Moreover, 
\[ h(\ell) = g_\ell(\ell^2) =
\Bigl( 1 - \frac{1}{\ell - 1}\Bigr)^{\ell + 1} \]
is an increasing function of $\ell$ for $\ell \ge 2$, tending to $1/\mathrm{e}$
as $\ell \rightarrow \infty$. Calculation shows that $h(9) \ge 1/4$,
and hence
$g_\ell(n) \ge 1/4$
whenever $9 \le \ell \le \sqrt{n}$. We can therefore
use the modified Spider Interrogation
Strategy to prove the following conjecture,
\emph{subject to the extra hypothesis that} $\ell \le \sqrt{n}$.

\begin{conjecture}\label{conj:prob}
Let $s \le \ell < n/2$.
There is a questioning strategy which, provided $\ell$ is sufficiently
large, guarantees to use at
most $n + \ell - 1$ questions to find all identities in an $n$-person room
containing $s$ spies, and will on average use at most
$n + 3 \ell / 4$ questions.
\end{conjecture}

The game-playing setting for Conjecture~\ref{conj:prob} is
the variant form of 
`Knights and Spies', in which the numbers of the spies
are randomly chosen at the start of the game,
and the Secret-Keeper's only responsibility is to decide
on their answers. Note that the information that exactly
$s$ spies are present is \emph{not} revealed to the Interrogator, and
need not by honoured by the Secret-Keeper when refuting a claim.
A similar conjecture, in which 
the number of spies was itself a random quantity $\le \ell$,
could also be stated.

The numerical results presented in the following
section suggest that Conjecture~\ref{conj:prob}
also holds when $\ell$ takes its largest possible
 value of $\lfloor (n-1)/2 \rfloor$.

\vspace{0pt}

\subsection{The Chain Building Strategy}
A chain of people, each of its members supporting
the next person along, is almost as valuable a configuration as the cycle
potentially created by the previous strategy.
Any such chain consists of a number (possibly zero) of spies, 
followed by a number (again possibly zero) of knights.
There are~\hbox{$k+1$} possible configurations for a chain
of length~$k$. Provided we have a knight to hand,
its members can be identified using
repeated
bisection
in a mere~$\lfloor \log_2 k \rfloor + 1$ 
questions; this meets the theoretical minimum for binary questions.
An example is shown in Figure~8 below.


\begin{figure}[h!]
\begin{center}
\includegraphics{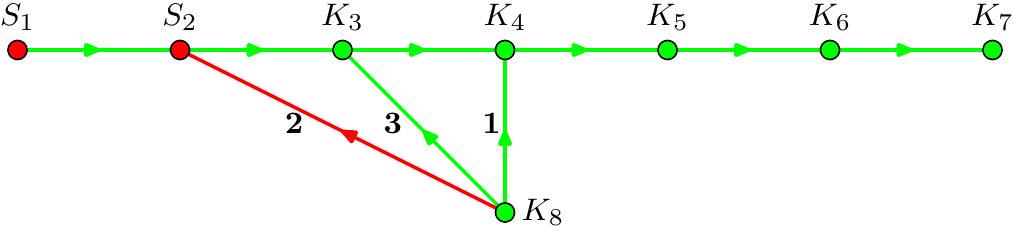}
\end{center}
\caption{Person $8$ is known to be a knight. Three questions
to him suffice to find all identities in the chain formed by persons $1$ to~$7$.}
\end{figure}

We now give a  rough outline of the \emph{Chain Building Strategy},
in which these chains play a fundamental role.
In the first step of the Chain Building Strategy
we hunt for someone who we 
can guarantee is a knight by first building chains,
starting a new chain as soon as we meet an accusation.
We then recursively link these chains by asking further questions 
(targeting people with the most persuasive support so far)
and stopping as soon as we reach someone who must
be a knight.
In the second step we use our guaranteed knight
to find everyone else's identity, exploiting the existing
chains as much as possible. 
An example of the critical first step is shown in Figure~9 above.

\begin{figure}[t]
\begin{center}
\includegraphics{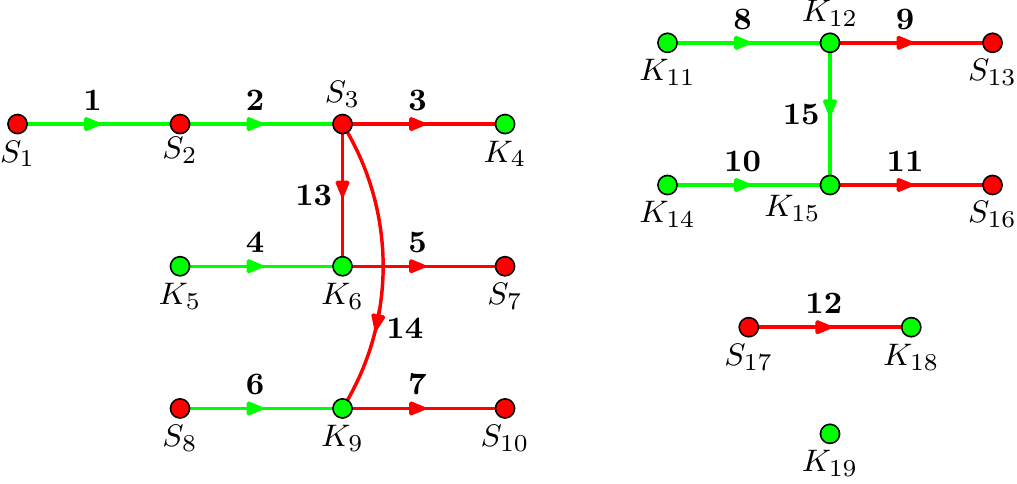}
\end{center}
\caption{The question graph after the first step of the Chain Building
Strategy in a room with $10$ knights and 
$9$ spies. Spies act knavishly, with the exception of
$S_{8}$, who we suppose answers truthfully when asked about $K_{9}$.
As in Step~1 of the Spider Interrogation Strategy,
the members of the components containing~$S_{1}$ and~$S_{17}$ are disregarded
once it becomes clear (after questions~14 and~12 respectively)
that they contain at least as many spies as knights.
The first step ends after question~$15$,
after which we can be sure that person~$K_{15}$ is a knight. In Step~2,
he will first be asked about $S_3$, bisecting
the longest chain.} 
\end{figure}

Simulation---both by hand, and by computer\footnote{Objective-C
source code for a program capable of simulating all the questioning strategies
discussed in this paper is available from 
\url{http://www.maths.bris.ac.uk/~mazmjw}.}---of the Chain Building Strategy
strongly suggests that, 
provided the behaviour of the spies is constrained in some way, 
or randomised entirely, 
it never requires more than $n + \ell - 1$ questions to find everyone's
identity. The numerical evidence also
suggests that Chain Building Strategy requires on average about
$4n/3$ questions to deal with a room in which knights are only just
in the minority, more than meeting the requirements of 
Conjecture~\ref{conj:prob}. Sadly, it appears that when~$\ell$ 
is a smaller fraction of $n$, for example, $\ell = n/4$, the
strategy is less effective. Some of the relevant
data is presented in Figures~10 and~11 overleaf.

At the time of writing, these
intermediate values for $\ell$ seem to present the largest
obstacle in the path to a proof of Conjecture~\ref{conj:prob}.

\begin{figure}[p]
\begin{center}
\includegraphics{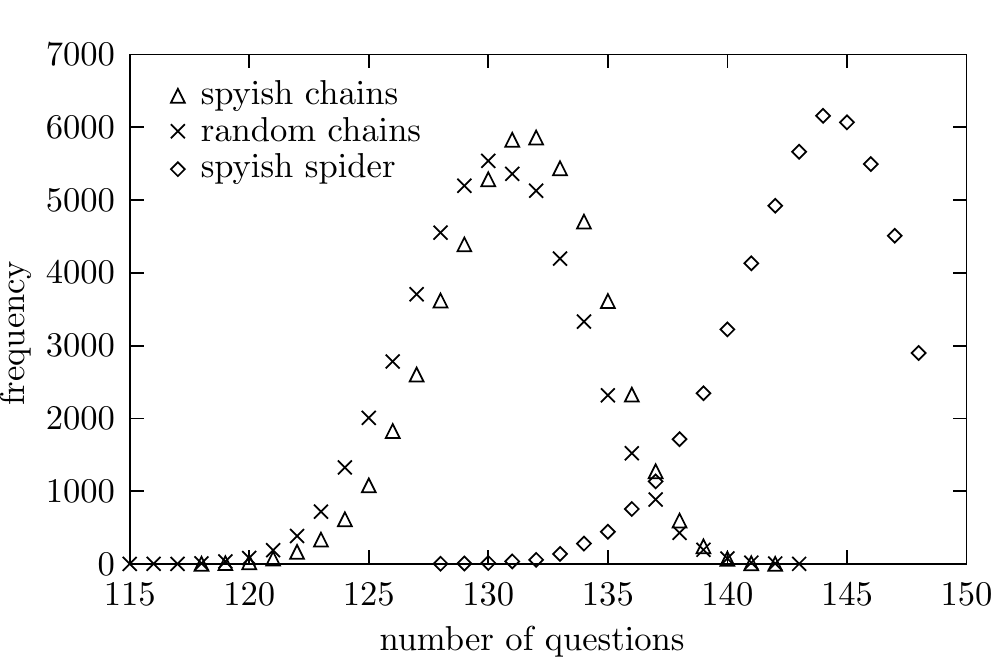}

\vspace{-6pt}
\end{center}
\caption{Numbers of questions asked
in $25000$ runs of the Chain Building Strategy in random generated rooms
with $51$ knights and $49$ spies. Results for spyish spies, and spies
which answer randomly are shown.
One might expect the Chain Building Strategy to fare significantly
worse when faced with spyish spies, since this behaviour certainly makes
it harder to form long chains. However, the difference is surprisingly
unpronounced, perhaps because the spies are prone to give themselves
away by their excessive accusations. 
%
For comparison, the corresponding
results obtained from simulation of the Spider Interrogation Strategy 
with spyish spies
are also shown.
}
\end{figure}

\begin{figure}[p]
\begin{center}
\includegraphics{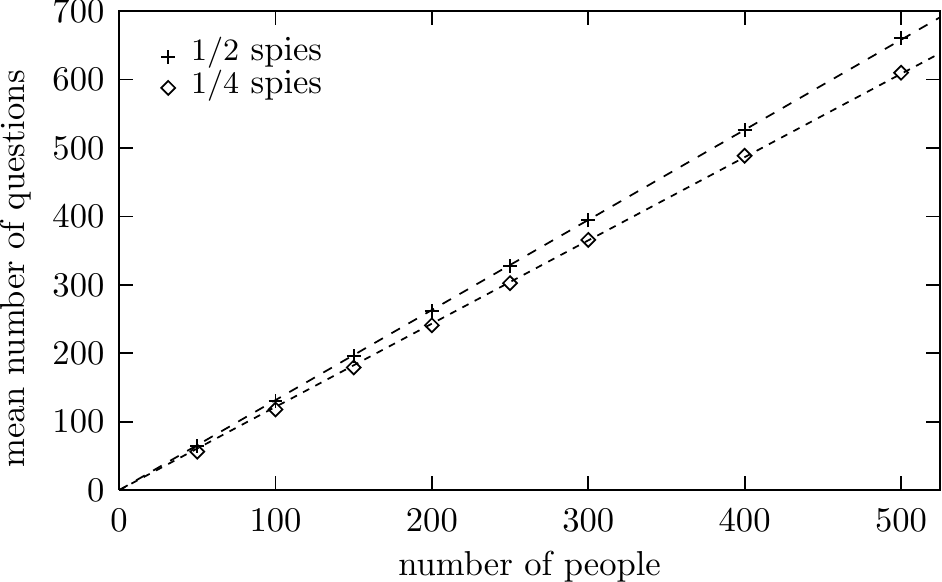}
\vspace{-6pt}
\end{center}
\caption{Mean number of questions asked in $1000$ runs
of the Chain Building Strategy in randomly generated rooms with $n$
people when $\ell = \lfloor (n-1)/2 \rfloor$ and $\lfloor n/4 \rfloor$ respectively.
In each case $\ell$ spies were present.
The gradients of the interpolating lines are $1.316$ and $1.217$ respectively.
Spies answered spyishly;
other constraints on their behaviour gave the same linear behaviour, with
similar gradients.}
\end{figure}


\section{Open problems and variant games}
We end by presenting two further open problems, which seem worthy
of attention, and may well
be more tractable than Conjecture~\ref{conj:prob}. 

\begin{problem}\label{prob:prop}
In a $4k$ person room known to contain exactly $k$ knights,
what is the smallest number of questions that will give
a probability $\ge 1/5$ of correctly identifying
every person?
\end{problem}

Now that we have dropped our long-standing assumption
that  spies are in the minority, we can no longer 
guarantee to find everyone's identity. 
However,
there is still a chance of success. Indeed, if we ask
all $n(n-1)$ useful questions, then the spies must
be careful not to give themselves away by forming
a block of $> k$ people, all of its members supporting one another.
Instead, in the worst case we are left with
four 
camps each of $k$ people, each camp
behaving as if they are the 
knights, and the opposing camps are the spies.
Choosing a camp
at random gives a~$1/4$ chance of success.   
Problem~\ref{prob:prop} asks whether, 
if we accept a smaller chance of success,
we might be able to manage with significantly fewer questions.

\begin{problem}\label{prob:single}
Let $\ell < n/2$.
In an $n$ person room with at most $\ell$ spies, what
is the smallest number of questions that will guarantee
to find at least one person's identity? What is the
smallest number of questions that will guarantee 
to find a knight?
\end{problem}

For example, given the sequence of questions shown in
Figure~1, we can be sure after question 15 that person~15
is a knight (and also that person~18 is a spy), but before
this question we cannot be certain of any single identity.

The Spider Interrogation Strategy shows that $2\ell - 1$ questions
suffice to find a knight. This gives an upper bound for both
parts of Problem~\ref{prob:single}. For a lower bound, it 
is natural to pose the problem in the 
game-playing framework
of~\S 3. The Mole Hiding
Strategy shows that~$\ell$ questions are necessary,
but cannot otherwise be recommended,
for if the Interrogator
follows the Spider Interrogation Strategy, then after she has asked
these~$\ell$ questions, she will be able to claim.


The author conjectures that the answer to the first part---and 
hence to both parts---of
Problem~\ref{prob:single} is $2\ell - 1$. If so, we
face the remarkable situation that, while we can find the identity
of a particular person, nominated in advance, with~$2\ell$ questions, we can only save one question
if the person is entirely of our choosing, to be nominated later.

In his famous
`\negvthinspace\emph{A Mathematician's Apology}' \cite[\S 15--17]{HardyApology},
G.~H.~Hardy argued that serious mathematics could be distinguished by virtue
of its
depth and generality, and also by a certain `\negvthinspace\emph{unexpectedness},
combined with \emph{inevitability} and \emph{economy}' (his emphasis).
The reader who has read this far will, it is hoped, agree
that the Knights and Spies Problem deserves to qualify under 
all of his criteria.

\medskip
\end{document}